\documentclass[11pt,twoside,a4paper]{article}%
\usepackage[margin=2.0cm]{geometry}
\usepackage{amsmath}
\usepackage{graphicx}
\usepackage{subfigure}
\usepackage{enumitem}
\usepackage{tikz}
\usepackage{fixltx2e} % to use command \textsuperscript
\usepackage{textcomp}
\usepackage{float}
\usepackage{setspace}
\usepackage{nth}
\usepackage{accents}
\usepackage{titlesec}
\usepackage{placeins}

\titleformat{\paragraph}
{\normalfont\normalsize\bfseries}{\theparagraph}{1em}{}
\titlespacing*{\paragraph}
{0pt}{3.25ex plus 1ex minus .2ex}{1.5ex plus .2ex}

\titlespacing*{\section}
{0pt}{5.5ex plus 1ex minus .2ex}{4.3ex plus .2ex}
\titlespacing*{\subsection}
{0pt}{5.5ex plus 1ex minus .2ex}{4.3ex plus .2ex}

\usepackage{fancyhdr}
\usepackage{nth}
\usepackage[nottoc]{tocbibind}
\usepackage{breqn}
\usepackage{amsthm}

\newcommand{\mb}[1]{\mathbf{#1}}
\newcommand{\td}[1]{\tilde{#1}}
\newcommand{\mcal}[1]{\mathcal{#1}}
\newcommand{\mbb}[1]{\mathbb{#1}}

\newcommand{\lan}{\left\langle}
\newcommand{\ran}{\right\rangle}

\newcommand{\pd}{\partial}	%partial derivative symbol
\newcommand{\tnsr}[2]{\ensuremath{{#1}_{i_1\dots i_{#2}}}}

\newtheorem{theorem}{Theorem}[section]

\newtheorem{exmp}{Example}[section]
\newtheorem{definition}{Definition}[section]

\theoremstyle{remark}
\newtheorem{remark}{Remark}

\usepackage{amsfonts}
\usepackage{amssymb}
\usepackage{longtable}
\usepackage{tensor}
\usepackage[numbers,sort&compress]{natbib}
\usepackage[toc,page]{appendix}
\usepackage{color}   %May be necessary if you want to color links
\usepackage{hyperref}
\def\equationautorefname~#1\null{ (#1)\null}
\hypersetup{
  colorlinks = true,
  citecolor = blue
  }
\titlespacing*{\section}
{0pt}{5.5ex plus 1ex minus .2ex}{4.3ex plus .2ex}
\titlespacing*{\subsection}
{0pt}{5.5ex plus 1ex minus .2ex}{4.3ex plus .2ex}

\title{Stable boundary conditions for the Hermite Discretization of the Boltzmann Equation in Multi Physical Space Dimensions}
\author{Neeraj Sarna\\Center for Computational Engineering \&\ Department of Mathematics\\RWTH\ Aachen University, Germany\thanks{Mathematics (CCES), Schinkelstr. 2,
52062 Aachen, Germany}\\{\small (\texttt{sarna@mathcces.rwth-aachen.de}
)}}
\date{(2016)}
\date{}

\begin{document}
\maketitle
\begin{abstract}
Any numerical method fails to provide us with acceptable results if not equipped with appropriate boundary conditions. Catering to more realistic applications, in the present article we have extended the work done in \cite{Sarna2017} to the Boltzmann equation involving multi-dimensions in physical and velocity space. Criteria for stable boundary conditions, using energy estimates, have been discussed for linear symmetric hyperbolic initial boundary value problems. Since the use of energy estimates requires the hyperbolic system to be symmetric, the symmetric hyperbolicity of the moment equations arising from a Hermite discretization of the Boltzmann equation has been studied. Furthermore, an algorithm to construct a general symmetrizer for an arbitrary order Hermite discretization has been presented. Similar to \cite{Sarna2017}, a block structure for the multi-dimensional moment equations has been recognised which has been used to construct stable Onsager boundary conditions. The newly proposed Onsager boundary conditions have been used to study a Poisson heat conduction problem using a higher order Hermite discretization; the results have been compared to those presented in \cite{Torrilhon2015}.
\end{abstract}

\section{Stable boundary conditions}
The present section presents the criteria which a set of boundary conditions, for linear symmetric hyperbolic initial boundary value problem(IBVPs), should satisfy in order to be stable; these criteria have also been presented in \cite{Nordstram2016,Friedrich1958,Sarna2017} and they originate from energy estimates. The energy estimates provide us with a upper bound for the solution in some norm. The stability of the boundary conditions is closely connected to the well-posedness of IBVPs, see \cite{Nordstram2016,Friedrich1958,David} for a detailed discussion on well-posedness. 

\subsection{Preliminaries}
A general linear IBVP can be given as 
\begin{subequations}
\begin{align}
\pd_t \boldsymbol{\alpha}(\mb{x},t) + \displaystyle \sum_{i = 1}^d\mb{A}^{(i)}\pd_{x_i} \boldsymbol{\alpha}(\mb{x},t) = &\mb{F}(\mb{x},t),  \quad  \quad \forall \mb{x} \in \Omega	\label{moment system abstract} \\
\boldsymbol{\alpha}(\mb{x},0) =& \mb{f}(\mb{x}) \label{initial condition}\\
\mb{B}\boldsymbol{\alpha}^{(n,t,r)} =& \mb{g}(t),\quad\quad \forall \mb{x} \in \pd \Omega \label{boundary condition}
\end{align}
\end{subequations}
where $\boldsymbol{\alpha}\in \mbb{R}^{m}$ is the solution vector, $\mb{A}^{(i)}\in \mbb{R}^{m\times m}$ is a constant coefficient matrix which is not necessarily symmetric. The boundary conditions are prescribed by $\mb{B}\in \mbb{R}^{p\times m}$, the exact form of which will be discussed in the coming sections, and $d$ represents the total number of spatial dimensions. The vectors $\mb{f}(\mb{x})\in \mbb{R}^m$ and $\mb{g}(t)\in \mbb{R}^{p}$ are the given data of the problem and represent the initial and the boundary conditions respectively; furthermore, $\mb{F}\in \mbb{R}^m$ is an external forcing applied to the system and can be used to drive the system into a particular direction. The vectors $\mb{f}$, $\mb{g}$ and $\mb{F}$ will be considered to be infinitely differentiable i.e. $\mb{f}\in \left[C^{\infty}(\Omega)\right]^m, \mb{F}\in\left[C^{\infty}(\Omega)\right]^m $ and $\mb{g}\in \left[C^{\infty}(\Omega)\right]^p$. 

If we represent with $\mb{n},\mb{t}$ and $\mb{r}$ the unit vectors which span the local coordinate system at a boundary point then $\boldsymbol{\alpha}^{(n,t,r)}$ represents the solution in this local coordinate system. The solution $\boldsymbol{\alpha}$, defined in the global coordinate system is related to $\boldsymbol{\alpha}^{(n,t,r)}$ by the following relation
\begin{align}
\boldsymbol{\alpha}^{(n,t,r)} = \mb{T}\boldsymbol{\alpha}
\end{align}
where $\mb{T}$ is a projector matrix. We will also assume the system in \eqref{moment system abstract} to be symmetric hyperbolic, the assumption of symmetric hyperbolicity is crucial for the application of energy estimates. Then due to the symmetric hyperbolicity of the system there will exist a symmetric positive definite matrix $\mb{S}$ such that it symmetrizes the system, in \eqref{moment system abstract}, from the left. We now have the following definition for stable boundary conditions (see \cite{Nordstram2016,Friedrich1958,Sarna2017}).

\begin{definition}
For an IBVP, a set of boundary conditions is said to be stable if it leads to the following energy estimate
\begin{align}
\|\boldsymbol{\alpha}(.,t)\|^2_{\mb{S}} \leq \lambda(t)\left(\|\mb{f}\|^2 + \int_0^t\left(\|\mb{F}(.,\tau)\|^2+|\mb{g}(\tau)|^2\right)d\tau\right)	\label{bound inhomo}
\end{align}
where $\lambda(t)$ is a function bounded independently of $\mb{f}(\mb{x})$ and $\mb{g}(t)$. The norm $\|\boldsymbol{\alpha}(.,t)\|_{\mb{S}}$ is defined as
\begin{align}
\|\boldsymbol{\alpha}\|_{\mb{S}} = \sqrt{\int_{\Omega}\boldsymbol{\alpha}^T\mb{S}\boldsymbol{\alpha}d\mb{x}}
\end{align}
\end{definition}

Due to the symmetric hyperbolic nature of the equations in \eqref{moment system abstract}, the quantity $\boldsymbol{\alpha}^T\mb{S}\boldsymbol{\alpha}$ represents a convex entropy functional for \eqref{moment system abstract}. Therefore the expression in \eqref{bound inhomo} means that we would like to prescribe the boundary conditions such that the temporal evolution of the $L^2(\Omega)$ norm of the entropy functional remains bounded by the given data of the problem. In all the coming analysis we will ignore the influence from external forcing and thus consider $\mb{F} = 0 $.

\begin{remark}
In the present work we are not concerned with discontinuous solutions therefore, in addition to being infinitely differentiable, we will consider the initial and the boundary conditions in \eqref{moment system abstract} to be compatible.
\end{remark}
\begin{remark}
To obtain a unique solution for our IBVP, it is crucial to prescribe appropriate number of boundary conditions. This translates into an appropriate value for $p$.
\end{remark}
\subsection{Symmetrizing the system of equations}
Let a variable $\mb{v} \in \mbb{R}^m$ be defined as 
\begin{align}
\mb{v} = \mb{S}^{\frac{1}{2}}\boldsymbol{\alpha}	\label{symm variable}
\end{align}
Then inserting the above relation in \eqref{moment system abstract}-\eqref{boundary condition} we obtain the following IBVP for $\mb{v}$
\begin{subequations}
\begin{align}
\pd_t \mb{v}(\mb{x},t) + \displaystyle \sum_{i = 1}^d\mb{S}^{\frac{1}{2}}\mb{A}^{(i)}\mb{S}^{-\frac{1}{2}}\pd_{x_i} \mb{v}(\mb{x},t) = &\mb{S}^{\frac{1}{2}}\mb{F}(\mb{x},t),  \quad  \quad \forall \mb{x} \in \Omega	\label{moment system abstract v} \\
\mb{v}(\mb{x},0) =& \mb{S}^{\frac{1}{2}}\mb{f}(\mb{x}) \label{initial condition v}\\
\mb{B}\mb{T}\mb{S}^{-\frac{1}{2}}\mb{v} =& \mb{g}(t),\quad\quad \forall \mb{x} \in \pd \Omega \label{boundary condition v}
\end{align}
\end{subequations}
Since $\mb{S}$ is a symmetrizer for our system in \eqref{moment system abstract} so all the matrices $\mb{S}^{\frac{1}{2}}\mb{A}^{(i)}\mb{S}^{-\frac{1}{2}}$ will be symmetric. 

\subsection{Rotational Invariance}
In addition to assuming symmetric hyperbolicity for \eqref{moment system abstract}, we will also assume the system to be rotationally invariant. Let $\mb{n}$ represent a unit vector then we can define $\mb{A}^{(n)}$ as 
\begin{align}
\mb{A}^{(n)} = \displaystyle \sum_{i = 1}^d\mb{A}^{(i)}n_i
\end{align}
Due to the assumption of rotational invariance, we have
\begin{align}
\mb{A}^{(n)} = \mb{T}^{-1} \mb{A}^{(1)} \mb{T},\quad \mb{T}^{-T}\mb{S}\mb{T}^{-1} = \mb{S} \label{rotatinal invariance}
\end{align}
Since $\mb{S}^{\frac{1}{2}}\mb{A}^{(n)}\mb{S}^{-\frac{1}{2}}$ is a symmetric matrix, so it's eigenvalue decomposition can be given as
\begin{gather}
\mb{S}^{\frac{1}{2}}\mb{A}^{(n)}\mb{S}^{-\frac{1}{2}} =  \mb{X}\mb{\Lambda}\mb{X}^{T}	\label{EV decomp}
\end{gather}
The similarity of $\mb{A}^n$ with respect to $\mb{A}^{(1)}$ shows us that the characteristic velocities of our system in \eqref{moment system abstract} and \eqref{moment system abstract v} are independent of the unit vector $\mb{n}$ and thus independent of the direction.

\subsection{Energy estimate}
To obtain an energy estimate for the symmetrized system in \eqref{moment system abstract v}, we multiply it from the left by $\mb{v}^T$, integrate over $\Omega$ and use the Gauss theorem to obtain 
\begin{align}
\pd_t \|\mb{v}\|^2 + \oint_{\pd \Omega}\mb{v}^T \mb{S}^{\frac{1}{2}}\mb{A}^n \mb{S}^{-\frac{1}{2}} \mb{v}d\mb{x} = 0.	\label{evolution norm}
\end{align}
We note that $\|\mb{v}\| = \|\boldsymbol{\alpha}\|_{\mb{S}}$ therefore the above relation governs the evolution of the entropy of our system. To obtain a bound of the form , we need to study the structure of the following quadratic form $\mcal{H}$
\begin{subequations}
\begin{align}
\mcal{H}=&\mb{v}^T \mb{S}^{\frac{1}{2}}\mb{A}^n \mb{S}^{-\frac{1}{2}} \mb{v}	\label{restriction boundary} \\
			= &\boldsymbol{\alpha}^T\mb{S}\mb{A}^n\boldsymbol{\alpha}	. \label{restriction boundary alpha}
\end{align}
\end{subequations}
We can now use the rotational invariance of our system to simplify the above expression for $\mcal{H}$ in the following way
\begin{subequations}
\begin{align}
\mcal{H} = &\boldsymbol{\alpha}^T\mb{S}\mb{A}^n\boldsymbol{\alpha} \\
			 = &\left(\boldsymbol{\alpha}^{(n,t,r)}\right)^T\mb{T}^{-T}\mb{S}\mb{T}^{-1}\mb{A}^{(1)}\boldsymbol{\alpha}^{(n,t,r)} \\
			 = &\left(\boldsymbol{\alpha}^{(n,t,r)}\right)^T\mb{S}\mb{A}^{(1)}\boldsymbol{\alpha}^{(n,t,r)}. \label{def H local coords}
\end{align}
\end{subequations}
We can now define characteristic variable, $\mb{W}$, for our system in \eqref{moment system abstract v} as
\begin{align}
\mb{W} = \mb{X}^T\mb{v}	\label{def W}
\end{align}
Using the above relation, we can transform \eqref{restriction boundary} to 
\begin{align}
\mcal{H} = \mb{W}^T\mb{\Lambda}\mb{W} = \mb{W}_{-}^T\mb{\Lambda}_{-}\mb{W}_- + \mb{W}_{+}^T\mb{\Lambda}_{+}\mb{W}_+	\label{def H}
\end{align}
where $\mb{W}_{-/+}$ are the characteristic variables which move with negative and positive characteristic speeds respectively. Additionally, $\boldsymbol{\Lambda}_{-/+}$ are diagonal matrices which collect negative and positive characteristic velocities on the diagonal. 

%\subsection{Homogeneous boundary conditions}
% In order to prescribe value to only the characteristic variables which come into the domain, we would like to have the following relation for $\mb{W}_-$ at the boundary in the homogeneous case
%%
%\begin{align}
%\mb{W}_- = \mb{R}_+ \mb{W}_+ + \mb{R}_0 \mb{W}_0	\label{W at boundary}
%\end{align}
%where $\mb{R}_+$ and $\mb{R}_0$ are given as (see \cite{Sarna2017})
%%
%%
%\begin{align}
%\mb{R}_0 =- \left(\mb{B}\mb{X}^{(1)}_-\right)^{-1}\mb{B}\mb{X}^{(1)}_0 ,\quad \mb{R}_+ =  -\left(\mb{B}\mb{X}^{(1)}_-\right)^{-1}\mb{B}\mb{X}^{(1)}_+ \label{def R0}
%\end{align}
%
%The above relations for $\mb{R}_{0/+}$ are only dependent upon the intial matrix $\mb{A}^{(i)}$ and have no dependence upon the symmetrizer $\mb{S}$. Substituting \eqref{W at boundary} in \eqref{def H} we obtain the following two conditions for stable boundary conditions (see \cite{Sarna2017,Torrilhon2017} for a complete derivation)
%%
%\begin{align}
%ker\{\mb{A}^{(1)}\} \subseteq ker\{\mb{B}\},\quad \quad \mb{R}_+^T\mb{\Lambda}_-\mb{R}_++\mb{\Lambda}_+ \geq 0	\label{stability condition homo}
%\end{align}

\subsection{Stability criteria}
 In order to prescribe values to only the characteristic variables which come into the domain, we would like to have the following relation for $\mb{W}_-$ at the boundary
\begin{align}
\mb{W}_- = \mb{R}_+ \mb{W}_+ + \mb{R}_0 \mb{W}_0 + \left(\mb{B}\mb{X}_-\right)^{-1}\mb{g}.	\label{W at boundary inhomo}
\end{align}
The matrices $\mb{R}_+$ and $\mb{R}_0$ appearing in the above expression are given as 
\begin{align}
\mb{R}_0 =- \left(\mb{B}\mb{X}_-\right)^{-1}\mb{B}\mb{X}_0 ,\quad \mb{R}_+ =  -\left(\mb{B}\mb{X}_-\right)^{-1}\mb{B}\mb{X}_+. \label{def R0}
\end{align}
See \cite{Torrilhon2017,Sarna2017} for more details. Substituting the above relation into \eqref{def H}, we obtain the following two conditions for a stable set of inhomogeneous boundary conditions(see \cite{Sarna2017,Torrilhon2017}) 
\begin{align}
ker\{\mb{A}^{(1)}\} \subseteq ker\{\mb{B}\},\quad \quad \mb{R}_+^T\mb{\Lambda}_-\mb{R}_++\mb{\Lambda}_+ > 0	\label{stability condition inhomo}
\end{align}
Since homogeneous boundary conditions are a special case of the inhomogeneous boundary conditions so we will not be discussing them in detail; see \cite{Nordstram2016,Friedrich1958,Torrilhon2017,Sarna2017} for a study of homogeneous boundary conditions.

\subsection{Onsager boundary conditions} \label{onsager boundary conditions}
Let us assume that the matrix $\mb{S}\mb{A}^{(1)}$ has the following structure 
\begin{align}
\mb{S}\mb{A}^{(1)} =  \left( \begin{matrix}\mb{0} &  \mb{A}^*\\
				\left(\mb{A}^*\right)^T&  0 	\end{matrix} \right)	\label{structure on An}
\end{align}
where $\mb{A}^*\in\mbb{R}^{p\times q}$ and $p + q  = m$. Additionally we will assume that $\mb{S}\mb{A}^{(1)}$ has the following properties
\begin{itemize}
\item The number of negative eigenvalues of $\mb{A}^{(1)}$ are equal to $p$.
\item The rows of $\mb{A}^*$ are linearly independent which leads to the following structure for $ker\{\mb{S}\mb{A}^{(1)}\}$
\begin{align}
ker\{\mb{S}\mb{A}^{(1)}\}=ker\{\mb{A}^{(1)}\} = 	\left( \begin{matrix} \mb{0}\\
									ker\{\mb{A}^*\} \end{matrix}
									 \right)	\label{structure kernel An}
\end{align}
\end{itemize}
In writing the first equality, we have used the fact that $\mb{S}$ is a symmetric positive definite matrix. Due to our assumption on the structure of $\mb{S}\mb{A}^{(1)}$ described in \eqref{structure on An}, we will consider $\boldsymbol{\alpha}$ and $\boldsymbol{\alpha}^{(n,t,r)} $ to be structured as 
\begin{align}
\boldsymbol{\alpha} = 	\left( \begin{matrix} \boldsymbol{\alpha}^p\\
									\boldsymbol{\alpha}^q \end{matrix}
									 \right),\quad 
									 \boldsymbol{\alpha}^{(n,t,r)} = 	\left( \begin{matrix} \boldsymbol{\alpha}_p^{(n,t,r)}\\
									\boldsymbol{\alpha}^{(n,t,r)}_q \end{matrix}
									 \right) \label{structure alpha}
\end{align}
where $\boldsymbol{\alpha}_p \in \mbb{R}^p$ and $\boldsymbol{\alpha}_q \in \mbb{R}^q$. Using \eqref{structure on An} and \eqref{structure alpha}, our quadratic form $\mcal{H}$ appearing in \eqref{def H local coords} can be simplified to 
\begin{align}
\mcal{H}  = 2\left(\boldsymbol{\alpha}^{(n,t,r)}_p\right)^T\mb{A}^*\boldsymbol{\alpha}^{(n,t,r)}_q	\label{def H Onsager}
\end{align}
Let us now relate $\boldsymbol{\alpha}^{(n,t,r)}_p$ to $\boldsymbol{\alpha}^{(n,t,r)}_q$, at the boundary, through the following relation
\begin{align}
\boldsymbol{\alpha}^{(n,t,r)}_p = \mb{L}\mb{A}^*\boldsymbol{\alpha}^{(n,t,r)}_q + \mb{g}	\label{Onsager boundary conditiones inhomo}
\end{align}
where $\mb{L}\in\mbb{R}^{p\times p}$ is a symmetric positive semi-definite Onsager matrix and $\mb{g}$ is the inhomogeneity arising from the wall. Due to our assumption on the number of negative eigenvalues of $\mb{A}^{(1)}$, the relation given in \eqref{Onsager boundary conditiones inhomo} prescribes the appropriate number of boundary conditions. Substituting the above boundary conditions into our quadratic form in \eqref{def H Onsager}, we obtain 
\begin{align}
\mcal{H} = 2\left(\boldsymbol{\alpha}^{(n,t,r)}_q\right)^T\left(\mb{A}^*\right)^T\mb{L} \mb{A}^*\boldsymbol{\alpha}^{(n,t,r)}_q  + 2\mb{g}^T \mb{A}^*\boldsymbol{\alpha}^{(n,t,r)}_q	\label{def H OBC inhomo}
\end{align}
In \cite{Sarna2017} it was observed that under certain assumptions on the inhomogeneity arising from the wall, $\mb{g}$ can be decomposed as 
\begin{align}
\mb{g} = \left( \begin{matrix} \mb{0}\\
					\hat{\mb{g}} \end{matrix} \right)	\label{structure g}
\end{align}
In the case of moment equations, to be discussed in the coming sections, $\mb{L}$ and $\mb{A}^*$ exhibit the following structure 
\begin{align}
\mb{A}^* =	 \left( \begin{matrix} 1 & \dots\\
						\mb{0} & \mb{A}^{\dagger} \end{matrix} \right)	\quad \mb{L} = \left( \begin{matrix} \mb{0} & \mb{0} \\
					\mb{0} & \hat{\mb{L}} \end{matrix} \right)\quad \boldsymbol{\alpha}^{(n,t,r)}_q = \left( \begin{matrix} \boldsymbol{\alpha}^{(n,t,r)}_q\\
					\hat{\boldsymbol{\alpha}}^{(n,t,r)}_q \end{matrix} \right)		\label{structure R}
\end{align}
where $\hat{\mb{L}}$ is a symmetric positive definite matrix. 
The above structure for $\mb{L}$ and $\mb{A}^*$ was identified in \cite{Sarna2017}. Using \eqref{structure g}  and \eqref{structure R}, $\mcal{H}$ appearing in \eqref{def H OBC inhomo} can be simplified to 
\begin{align}
\mcal{H} = 2\left(\hat{\boldsymbol{\alpha}}^{(n,t,r)}_q\right)^T\left(\mb{A}^{\dagger}\right)^T\hat{\mb{L}} \mb{A}^{\dagger}\hat{\boldsymbol{\alpha}}^{(n,t,r)}_q  + 2\hat{\mb{g}}^T \mb{A}^{\dagger}\hat{\boldsymbol{\alpha}}^{(n,t,r)}_q
\end{align}
The spd nature of $\hat{\mb{L}}$ provides us with the following statement
\begin{align}
\left(\hat{\boldsymbol{\alpha}}^{(n,t,r)}_q\right)^T\left(\mb{A}^{\dagger}\right)^T\hat{\mb{L}} \mb{A}^{\dagger}\hat{\boldsymbol{\alpha}}^{(n,t,r)}_q \geq 0 
\end{align}
 which immediately provides us with a bound of the form \eqref{bound inhomo}.

\section{The Boltzmann Equation}

Considering the physical and the velocity space to be multi-dimensional, the Boltzmann equation is given as 
\begin{align}
\partial_t f + \xi_i \partial_{x_i} f = Q(f) \label{BE}
\end{align}
where $\left(t,\mb{x},\boldsymbol{\xi}\right)\in \mbb{R}^+\times \mbb{R}^d\times \mbb{R}^d$, $Q(f)$ is the Boltzmann collision operator \cite{Carlos} and $f = f(\mb{x},\boldsymbol{\xi},t)$ defines the phase density function. In the above equation, a sum over $i$ from $1$ to $d$ has been implicitly assumed (summation convention). The collision operator, $Q(f)$ is such that 
\begin{align}
Q(f_{\mcal{M}}) = 0	\label{equilibrium condition}
\end{align}
where $f_{\mcal{M}}$ is the Maxwell-Boltzmann distribution function given as 
\begin{equation}
 f_{\mcal{M}}(\boldsymbol{\xi};\rho,\mb{v},\theta) = \frac{\rho(t,x)}{\left(2\pi \theta(t,x)\right)^{d/2}}\exp \left(-\frac{\left(\xi_i-v_i(t,x)\right)^2}{2\theta(t,x)}\right)	\label{MB Distribution}
\end{equation}
In the above equation, $\rho$, $\mb{v}$ and $\theta$ represent the density, velocity and temperature (in energy units) respectively of the flow and are defined with respect to $f$ as 
\begin{gather}
\rho = m \int_{\mbb{R}^d} f d\boldsymbol{\xi},\quad
\rho v_i = m \int_{\mbb{R}^d} \xi_i f d\boldsymbol{\xi},\quad
\rho v^2 + d\rho \theta = m \int_{\mbb{R}^d}\xi^2fd\boldsymbol{\xi}
\end{gather}

In the present work, we are only interested in flow states which lie close to the global equilibrium $f_0 = f_{\mcal{M}}(\boldsymbol\xi;\rho_0,0,\theta_0)$. Clearly , $\pd_t f_0 = \pd_{x_{i}}f_0 = 0$. The quantities $\rho_0$ and $\theta_0$ represent the corresponding ground states for density and temperature respectively. Let $\epsilon$ represent some smallness parameter then we can linearise $f$ about $f_0$ as 
\begin{align}
f = f_0 + \epsilon \td{f}.
\end{align}
Substituting the above linearisation into \eqref{BE} and using $\pd_t f_0 = \pd_{x_i}f_0 = 0$, we obtain the following equation upto $\mcal{O}(\epsilon)$
\begin{align}
\partial_t \td{f} + \xi_i \partial_{x_i} \td{f} = \td{Q}(\td{f}) \label{BE linear}
\end{align}
where $\td{Q}(\td{f})$ is some linearisation of $Q(f)$ about $f_0$. In writing the above equation we have used the fact that $Q(f_0) = 0$ which trivially follows from \eqref{equilibrium condition}. 
\subsection{Hermite Discretization}
Similar to \cite{Grad1949}, we will discretize $\td{f}$ as 
\begin{equation}
\td{f} \approx \td{f}_{h}\left(  \mathbf{x},\boldsymbol{\xi},t\right)  =\sum_{n=0}^{N_{d}}\sum
_{s=0}^{M_{n}}\alpha_{i_{1}i_{2}\cdots i_{n}}^{(s)}(\mathbf{x},t)\psi_{i_{1}%
i_{2}\cdots i_{n}}^{(s)}\left(\frac{\boldsymbol{\xi}}{\theta_{0}^{1/2}}\right)~f_0  \label{Ansatz}%
\end{equation}
where $(n,s)\in \mbb{N}\times \mbb{N}$. The values of $N_d$ and $M_n$ in the above expression represent our resolution in the velocity space. The values selected for $N_d$ and $M_n$ determine the moment theory being considered, see \cite{Torrilhon2015} for more details. The basis functions $\psi_{i_{1}i_{2}\cdots i_{n}}^{(s)}$ appearing in the above expression are given as 
\begin{align}
\psi^{(s)}_{i_1\dots i_n}(\boldsymbol\xi)  = L_s^{(n)}\left(\frac{\xi_i\xi_i}{2\theta_0}\right)\nu_{i_1\dots i_n} \label{def psi}
\end{align}
with $L^{(n)}_s$ and $\nu_{i_1\dots i_n}$ defined as 
\begin{subequations}
\begin{gather}
L_s^{(n)}(x) = 2^{n/2} x^{n/2} \sqrt{\frac{\Gamma \left(n+\frac{3}{2}\right)}{n! s! \Gamma \left(n+s+\frac{3}{2}\right)}} \sum_{p=0}^s (-1)^p\frac{\Gamma(n+s+3/2)}{\Gamma(n+p+3/2)}{{s}\choose{p}}x^p\label{laguerre1} \\
\nu_{i_1\dots i_n} = \nu_{\langle i_{1}}\nu_{i_{2}}\cdots\nu_{i_{n}\rangle}=\frac{(-1)^{n}%
}{(2n-1)!!}\left\Vert
%TCIMACRO{\TeXButton{nu}{{\boldsymbol\nu}}}%
%BeginExpansion
{\boldsymbol x }%
%EndExpansion
\right\Vert ^{n+1}\frac{\partial^{n}}{\partial x_{i_{1}}\partial x_{i_{2}%
}\cdots\partial x_{i_{n}}}\left(  \frac{1}{\left\Vert
%TCIMACRO{\TeXButton{xi}{{\boldsymbol\xi}}}%
%BeginExpansion
{\boldsymbol x}%
%EndExpansion
\right\Vert }\right)  \label{nuFormel}%
\end{gather}
\end{subequations}
where $\mb{x}$ is some position vector. As is clear from the above formulae, $L^{(n)}_s$ models the radial dependence of the distribution function and the trace free tensor $\nu_{i_1\dots i_n}$ models the anisotropy of the distribution function; see \cite{ET,Torrilhon2015,Torrilhon2017} for more details. The basis functions, $\psi^{(s)}_{i_1\dots i_n}(\boldsymbol\xi)$ , enjoy the following orthogonality property
\begin{align}
A_{i_1\dots i_n} \lan\psi^{(s)}_{i_1\dots i_n},\psi^{(r)}_{j_1\dots j_m} \ran_{\mbb{R}^{d},f_0}
= 
 \begin{cases} 
      0, & (n,s)\neq (m,r) \\
	 A_{\lan j_1\dots j_m \ran }, & \text{else}
   \end{cases}		\label{orthogonality psi}
\end{align}
Testing our discretization in \eqref{Ansatz} with $\psi^{(s)}_{i_1\dots i_n}$ and using the orthogonality of the basis functions given in \eqref{orthogonality psi}, we obtain the following relation for $\alpha_{i_1\dots i_n}^{(s)}$
\begin{align}
\alpha_{\lan i_1\dots i_n\ran}^{(s)} = \lan \td{f}_h,\psi^{(s)}_{i_1\dots i_n} \ran_{\mbb{R}^d,f_0} \label{def alpha}
\end{align}
Due to the above relation, we will consider all the $\alpha_{i_1\dots i_n}^{(s)}$ to be trace-free. The first few $\alpha_{i_1\dots i_n}^{(s)}$ appearing in \eqref{Ansatz} are related to the macroscopic quantities through the following relations 
\begin{subequations}
\begin{gather}
\frac{\tilde{\rho}}{\rho_0} = \alpha^{(0)},\quad \frac{\tilde{v}_i}{\sqrt{\theta_0}} = \alpha_i^{(0)},\quad \frac{\tilde{\theta}}{\theta_0}=-\sqrt{\frac{2}{3}}\alpha^{(1)} \label{relation alpha prim1}\\
\quad \frac{\tilde{\sigma}_{ij}}{\rho_0\theta_0} = \sqrt{2}\alpha_{ij}^{(0)},\quad \frac{\tilde{q}_i}{\rho_0\theta_0^{\frac{3}{2}}}=-\sqrt{\frac{5}{2}}\alpha_i^{(1)}	\label{relation alpha prim2}
\end{gather}
\end{subequations}
where $\td{\sigma}_{ij}$(stress tensor) and $\td{q}_i$(heat flux) represent the deviation of $\sigma_{ij}$ and $q_i$ from their respective ground states.
In \cite{Torrilhon2015} it was discussed that by appropriately choosing $N_d$ and $M_n$, we can ensure the rotational invariance of our moment system; therefore in the present work we will only be considering those moment systems which are rotationally invariant. Due to rotational invariance, the quadratic form $\mcal{H}$ appearing in the energy estimate will only have a contribution from $\mb{A}^{(1)}$ (see \eqref{def H local coords}).
In \autoref{onsager boundary conditions}, the formulation of OBCs for a general system relied upon certain crucial properties of $\mb{A}^{(1)}$. To identify these properties for our moment system, 
 it would be sufficient to divide the basis functions in \eqref{def psi} depending upon their even and odd property with respect to $\xi_1$.
 In addition to $\psi_{i_1\dots i_n}^{(s)}$, with $\psi_{i_1\dots i_n}^{(s,o)}$ and $\psi_{i_1\dots i_n}^{(s,e)}$ we will represent those basis functions which are odd and even in $\xi_1$ respectively. Due to the orthogonality of the basis functions \eqref{orthogonality psi}, we have the following relation for $\psi_{i_1\dots i_n}^{(s,o)}$ and $\psi_{i_1\dots i_n}^{(s,e)}$
\begin{align}
\lan\psi^{(s,o)}_{i_1\dots i_n},\psi^{(r,e)}_{j_1\dots j_m}\ran_{\mbb{R}^d,f_0} = 0\quad \quad \forall (n,m,s,r) \label{orthogonality even odd}
\end{align}
Using $\psi^{(s,o)}_{i_1\dots i_n}$  and $\psi^{(s,e)}_{i_1\dots i_n}$, we can now define the moments $\alpha^{(s,o)}_{i_1\dots i_n}$  and $\alpha^{(s,e)}_{i_1\dots i_n}$ as
\begin{align}
\alpha_{\lan i_1\dots i_n\ran}^{(s,e)} = \lan \td{f}_h,\psi^{(s,e)}_{i_1\dots i_n} \ran_{\mbb{R}^d,f_0}, \quad \alpha_{\lan i_1\dots i_n\ran}^{(s,o)} = \lan \td{f}_h,\psi^{(s,o)}_{i_1\dots i_n} \ran_{\mbb{R}^d,f_0} 	\label{def even odd moments}
\end{align}
With $n_o$ and $n_e$ we will represent the total number of odd and even moments respectively. We will now split $\td{f}_h$ into $\td{f}_h^o$ and $\td{f}_h^e$ in the following way
\begin{align}
\td{f}_h = \td{f}_h^o + \td{f}_h^e	\label{even odd split f}
\end{align}
where $\td{f}_h^o$ and $\td{f}_h^e$ are odd and even functions of $\xi_x$ respectively.
We note that $\td{f}_h^o\in span\{\psi^{(s,o)}_{i_1\dots i_n}f_0\}$ and $\td{f}_h^e\in span\{\psi^{(s,e)}_{i_1\dots i_n}f_0\}$.
\subsection{The Moment system}
Inserting our discretization in \eqref{Ansatz} into our linearised Boltzmann equation \eqref{BE linear} and integrating with respect to $\boldsymbol{\xi}$ after multiplication with $\psi^{(s)}_{i_1\dots i_n}$, we obtain the following expression 
\begin{align}
\pd_t \left(\lan \psi^{(s)}_{i_1\dots i_n} ,\td{f}^o_h  \ran_{\mbb{R}^d}\right. + &\left.\lan \psi^{(s)}_{i_1\dots i_n} ,\td{f}^e_h  \ran_{\mbb{R}^d}\right) + \nonumber\\
&\pd_{x_k} \left(\lan\xi_k \psi^{(s)}_{i_1\dots i_n} ,\td{f}^o_h  \ran_{\mbb{R}^d} + \lan \xi_k\psi^{(s)}_{i_1\dots i_n} ,\td{f}^e_h  \ran_{\mbb{R}^d}\right)
  = \lan\psi^{(s)}_{i_1\dots i_n},\td{Q}(\td{f})\ran_{\mbb{R}^d}.	\label{weak BE}
\end{align}

Similar to \cite{Sarna2017}, we will be ignoring the contribution from $\lan\psi^{(s)}_{i_1\dots i_n},\td{Q}(\td{f})\ran_{\mbb{R}^d}$ since it does not leads to any growth in $\|\boldsymbol{\alpha}\|$; see \cite{ET,Struchtrupbook,Carlos} for more details.
Due to the orthogonality of the even and odd basis functions given in \eqref{orthogonality even odd} and the recursion relations for the Laguerre polynomials (see \cite{Torrilhon2015}), we have the following relations
\begin{subequations}
\begin{gather}
\lan \psi^{(s,o)}_{i_1\dots i_n} ,\td{f}^e_h  \ran_{\mbb{R}^d} = 0 ,\quad \lan \xi_1\psi^{(s,o)}_{i_1\dots i_n} ,\td{f}^o_h  \ran_{\mbb{R}^d} = 0 \label{OrthoConv1}\\
\lan \psi^{(s,e)}_{i_1\dots i_n} ,\td{f}^o_h  \ran_{\mbb{R}^d} = 0 ,\quad \lan \xi_1\psi^{(s,e)}_{i_1\dots i_n} ,\td{f}^e_h  \ran_{\mbb{R}^d} = 0. \label{OrthoConv2}
\end{gather}
\end{subequations}
Choosing $\psi^{(s)}_{i_1\dots i_n}$ to be $\psi^{(s,o)}_{i_1\dots i_n}$ and $\psi^{(s,e)}_{i_1\dots i_n}$ consecutively in \eqref{weak BE}, we obtain the following equations for the set of even, $\boldsymbol\alpha_e\in\mbb{R}^{n_e}$, and odd, $\boldsymbol{\alpha}_o\in\mbb{R}^{n_o}$, moments
\begin{align}
\pd_t \boldsymbol\alpha_o + \displaystyle\sum_{i = 1}^d\bar{\mb{A}}_{oe}^{(i)}\pd_{x_i}\boldsymbol{\alpha} = 0 ,\quad 
\pd_t \boldsymbol\alpha_e + \displaystyle\sum_{i = 1}^d\bar{\mb{A}}_{eo}^{(i)}\pd_{x_i}\boldsymbol{\alpha} = 0\label{moment system}
\end{align}
where $\bar{\mb{A}}_{oe}^{(i)}\in\mbb{R}^{n_o\times n_e} $ and $\bar{\mb{A}}_{eo}^{(i)}\in\mbb{R}^{n_e\times n_o}$. In the above relation, we have assumed $\boldsymbol{\alpha}$ to be ordered as $\boldsymbol{\alpha} = \left(\boldsymbol{\alpha}_o,\boldsymbol{\alpha}_e\right)^T$. Using the orthogonality relations from \eqref{OrthoConv1} and \eqref{OrthoConv2}, we can identify the following structure for $\mb{A}_{oe}^{(1)}$ and $\mb{A}_{eo}^{(1)}$
\begin{align}
\bar{\mb{A}}_{oe}^{(1)} = \left(\mb{0},\mb{A}_{oe}^{(1)}\right) \quad \bar{\mb{A}}_{eo}^{(1)} = \left(\mb{A}_{eo}^{(1)},\mb{0}\right).	\label{structure Aoe}
\end{align}
It is crucial to note that all the other matrices appearing in our moment system, apart from $\bar{\mb{A}}_{oe}^{(1)}$ and $\bar{\mb{A}}_{oe}^{(1)}$, will not have the same structure as given in \eqref{structure Aoe}. This is due to the fact that the basis functions $\psi_{i_1\dots i_n}^{(s,e)}$ and $\psi_{i_1\dots i_n}^{(s,o)}$ are only even and odd with respect to $\xi_1$;  as a result of which the terms $\lan \xi_{k}\psi^{(s,o)}_{i_1\dots i_n} ,\td{f}^o_h  \ran_{\mbb{R}^d}$ and $\lan \xi_k\psi^{(s,o)}_{i_1\dots i_n} ,\td{f}^o_h  \ran_{\mbb{R}^d}$ for all $k\in\{2,3\}$ do not necessarily vanish. Using the structure of $\bar{\mb{A}}_{oe}^{(1)}$ and $\bar{\mb{A}}_{eo}^{(1)}$, the matrix $\mb{A}^{(1)}$ appearing in our general setting \eqref{moment system abstract} which corresponds to our moment system will have the following structure
\begin{align}
\mb{A}^{(1)} =  \left( \begin{matrix}\mb{0} &  \mb{A}_{oe}^{(1)}\\
				\mb{A}_{eo}^{(1)}&  0 	\end{matrix} \right).	\label{structure on A1}
\end{align}
To formulate stable boundary conditions for our system in \eqref{moment system} , using energy estimates, we will now show that our system in \eqref{moment system} is symmetric hyperbolic and we will also discuss a methodology to construct a symmetrising matrix, $\mb{S}$, for a general moment system.
\subsubsection{Symmetric Hyperbolicity}
In our discretization \eqref{Ansatz}, the moments $\alpha_{i_1\dots i_n}^{(s)}$ have been considered to be trace free due to the orthogonality property of the basis functions \eqref{orthogonality psi}. The trace-free nature of the moments being considered reduces the size of our solution vector $\boldsymbol{\alpha}$. For e.g. if a second order tensor is considered to be trace-free then the total number of unknowns are reduced from nine to five. But let us consider a situation where we consider all the components of every tensor; so an $n$-th order tensor will have in total $3^n$ components. In such a case, our moment system can be generically represented as 
\begin{align}
\pd_t \bar{\boldsymbol{\alpha}}(\mb{x},t) + \displaystyle \sum_{i = 1}^d\bar{\mb{A}}^{(i)}\pd_{x_i} \bar{\boldsymbol{\alpha}}(\mb{x},t) = \mb{0},  \quad  \quad \forall \mb{x} \in \Omega
\end{align} 
where $\bar{\boldsymbol{\alpha}}$ is a solution vector which contains all the components of all the moments begin considered. Obviously, $\bar{\boldsymbol{\alpha}}$ and $\bar{\mb{A}}^{(i)}$ will be bigger in dimension than $\boldsymbol{\alpha}$ and $\mb{A}^{(i)}$. We claim that $\bar{\mb{A}}^{(i)}$ will be symmetric matrices. To show this, we integrate the following identity
\begin{subequations}
\begin{align}
\partial_{\xi_k}\left(\tnsr{\psi}{m}^{(r)} f_0 \tnsr{\psi}{n}^{(s)} \right) = &\left(\partial_{\xi_k}\tnsr{\psi}{m}^{(r)}\right) f_0 \tnsr{\psi}{n}^{(s)} + \tnsr{\psi}{m}^{(r)} \left(\partial_{\xi_k}f_0\right) \tnsr{\psi}{n}^{(s)}\nonumber\\ 
&+ \tnsr{\psi}{m}^{(r)} f_0 (\partial_{\xi_k}\tnsr{\psi}{n}^{(s)}) \\ 
= & \frac{1}{\theta_0}\tnsr{\psi}{m}^{(r)} \xi_k f_0 \tnsr{\psi}{n}^{(s)} + \tnsr{\psi}{m}^{(r)} \partial_{c_x}\left(f_0 \tnsr{\psi}{n}^{(s)}\right)\nonumber\\ + &\tnsr{\psi}{n}^{(s)} \partial_{\xi_k}\left(f_0 \tnsr{\psi}{m}^{(r)}\right)
\end{align}
\end{subequations}
Now using $\partial_{\xi_k}\left\langle \tnsr{\psi}{m}^{(r)}, f_0 \tnsr{\psi}{n}^{(s)} \right \rangle = 0$, we obtain
\begin{align}
\left\langle \tnsr{\psi}{m}^{(r)}, \xi_k f_0 \tnsr{\psi}{n}^{(s)}\right\rangle =- \theta_0\left\langle\tnsr{\psi}{m}^{(r)}, \partial_{\xi_k}\left(f_0 \tnsr{\psi}{n}^{(s)}\right)\right\rangle - \theta_0\left\langle\tnsr{\psi}{n}^{(s)}, \partial_{\xi_k}\left(f_0 \tnsr{\psi}{m}^{(r)}\right)\right\rangle \label{elongated moment system}
\end{align}

The above identity implies that $\left\langle \tnsr{\psi}{m}^{(r)}, \xi_k f_0 \tnsr{\psi}{n}^{(s)}\right\rangle$ is symmetric with respect to the pairs $(m,r)$ and $(n,s)$ for all values of $k$. Considering the derivation of the moment system presented in \eqref{weak BE}, we see that the matrices $\bar{\mb{A}}^{(k)}$ are nothing but $\left\langle \tnsr{\psi}{m}^{(r)}, \xi_k f_0 \tnsr{\psi}{n}^{(s)}\right\rangle$ placed at appropriate locations with some ordering for the tuples $\{i_1\dots i_n\}$ and $\{i_1\dots i_n\}$. Therefore $\bar{\mb{A}}^{(i)}$ will be symmetric. But since the moments being considered are tracefree, a set of equations appearing in \eqref{elongated moment system} will be identical and so the system in \eqref{elongated moment system} can be reduced by removing these equations. This reduction of our system in \eqref{elongated moment system} will lead to our original system given in \eqref{moment system} but will rob $\bar{\mb{A}^{(i)}}$ of it's symmetricity. We note that the system of equations in \eqref{elongated moment system} and \eqref{moment system} is the same, with \eqref{moment system} being just a reduction of \eqref{elongated moment system} obtained by removing identical equations, therefore the hyperbolic nature of our equations will not be lost. Since our system in \eqref{elongated moment system} is symmetric thus it's convex entropy functional $\eta(\bar{\boldsymbol{\alpha}})$ will be given by
\begin{align}
\eta(\bar{\boldsymbol{\alpha}}) = \frac{1}{2}\bar{\boldsymbol{\alpha}}^T\bar{\boldsymbol{\alpha}} 	\label{entropy functional}
\end{align}
Due to the similarity between the systems in \eqref{elongated moment system} and \eqref{moment system}, they should have the same entropy functional. An entropy functional for our system in \eqref{moment system} could be found if we can express $\eta(\bar{\boldsymbol{\alpha}})$ in terms of $\boldsymbol{\alpha}$ with the help of a symmetric matrix $\mb{S}$
\begin{align}
\eta(\boldsymbol{\alpha}) = \boldsymbol{\alpha}^T\mb{S}\boldsymbol{\alpha}
\end{align}
Then $\mb{S}$ will be related to $\eta$ through the following relation
\begin{align}
\mb{S} = \frac{1}{2}\frac{\pd^2\eta}{\pd \boldsymbol{\alpha}^2}	\label{relation S to eta}
\end{align}
Clearly, the assumed convexity of $\eta\left(\boldsymbol{\alpha}\right)$ implies the positive definiteness of $\mb{S}$ due to the above relation. The matrix $\mb{S}$ will then symmetrize our system in \eqref{moment system} from the left, thanks to the following theorem
\begin{theorem}\label{entropy sym system}
If a hyperbolic system is endowed with a convex entropy functional $\eta(\boldsymbol{\alpha})$, then the following variable transformation symmetrizes the system in the Friedrich's sense 
\begin{gather}
\mb{v} = \mb{S}^{\frac{1}{2}}\boldsymbol{\alpha}	\label{sym transformation}
\end{gather}
where $\mb{S} =\frac{1}{2} \frac{\pd^2 \eta(\boldsymbol{\alpha})}{\pd \boldsymbol{\alpha}^2}$
\end{theorem}
\begin{proof}
See \cite{Tadmor}
\end{proof}
The above analysis shows us that our system in \eqref{moment system} is symmetric hyperbolic therefore we can use the method of energy estimate to construct stable boundary conditions for the same. 
\subsubsection{Symmetrizing matrix}
Before looking into the construction of the symmetrising matrix, let us consider a simple example. Let's assume we have a moment system which is symmetric and only consists of a second order trace-free tensor $R_{ij}$. For simplicity if we now consider a two-dimensional physical space, then $\bar{\boldsymbol{\alpha}}$ appearing in \eqref{elongated moment system} will have in total five components and will be given as 
\begin{align}
\bar{\boldsymbol{\alpha}} = \{R_{xy},R_{yx},R_{xx},R_{yy},R_{zz}\}
\end{align}
We note that even in the two-dimensional setting, we have considered $R_{zz}$ to be a part of the solution vector since it is related to $R_{xx}$ and $R_{yy}$ due to the trace free nature of $R_{ij}$. For such a system, the entropy functional appearing in \eqref{entropy functional} will be given as 
\begin{align}
\eta = \frac{1}{2}\left(R_{xx}^2 + R_{xy}^2 + R_{yy}^2 + R_{yx}^2 + R_{zz}^2\right)
\end{align}
Since $R_{ij}$ is trace-free, we can reduce $\bar{\boldsymbol{\alpha}}$ to $\boldsymbol{\alpha}$ which can be given as 
 \begin{align}
 \boldsymbol{\alpha} = \{R_{xy},R_{xx},R_{yy}\}
 \end{align}
Using the tracefree nature of $R_{ij}$, we can express $\eta$ as 
\begin{align}
\eta = \frac{1}{2}\left(2 R_{xx}^2 + 2R_{xy}^2 + 2R_{yy}^2 +  2 R_{xx}R_{yy}\right) = \boldsymbol{\alpha}^T \mb{S}\boldsymbol{\alpha} 
\end{align}
If we now use the relation between $\mb{S}$ and $\eta$ given in \eqref{relation S to eta}, then $\mb{S}$ can be identified as
\begin{align}
\mb{S} = \left( \begin{matrix} 1 & 0 & 0 \\ 
										 0 & 1 & \frac{1}{2}  \\
										 0 & \frac{1}{2} & 1   	\end{matrix} \right)
										 \label{S2}
\end{align}
The matrix $\mb{S}$ collects the coefficients which arise in the entropy functional due to the trace-free property of $R_{ij}$ and thus helps us in expressing $\eta$ in terms of our reduced variables $\boldsymbol{\alpha}$. For a general moment system, we can construct the matrix $\mb{S}$ by first constructing the sub-matrices $\mb{S}_n$ which collect the coefficients corresponding to a tensor of degree $n$. The full matrix $\mb{S}$ can then be developed by placing different entries of $\mb{S}_n$ at appropriate locations. 
\begin{exmp}
As an example, let us consider the Grad's-20 (G20) moment system which can be derived by considering the following values for $N_d$ and $M_n$ in our Hermite discretization \eqref{Ansatz}
\begin{align}
N_d = 3,\quad M_0 = 2,\quad M_1 = 2,\quad M_2 = 1,\quad M_3 = 1
\end{align}
see \cite{Torrilhon2015} for more details. If we now restrict ourselves to two and three dimensional physical and velocity space respectively then the vector $\boldsymbol{\alpha}$ is given as 
\begin{align}
\boldsymbol{\alpha} = \{\alpha_x^{(0)},\alpha_{xy}^{(0)},\alpha_x^{(1)},\alpha_{xxx}^{(0)},\alpha_{xyy}^{(0)},\alpha^{(0)},\alpha_y^{(0)},\alpha^{(1)},\alpha^{(0)}_{xx},\alpha_{yy}^{(0)},\alpha_y^{(1)},\alpha_{xxy}^{(0)},\alpha^{(0)}_{yyy}\}
\end{align}
As mentioned above, we will first construct the contributions from the different tensor degrees appearing in our moment set. From the solution vector given in the above expression, we find that the G20 moment system consists of four different tensor degrees i.e. $n = \{0,1,2,3\}$. Therefore we need $S_0$, $\mb{S}_1$ , $\mb{S}_2$ and $\mb{S}_3$ to fully define our symmetrising matrix $\mb{S}$.
The matrix corresponding to the second order tensor, $\mb{S}_2$, has already been given in \eqref{S2}. The expressions for $S_0$, $\mb{S}_1$ and $\mb{S}_3$ are given as 
\begin{align}
S_0 = \frac{1}{2},\quad
 \mb{S}_1 = \left( \begin{matrix} \frac{1}{2} & 0 \\ 
											  0 & \frac{1}{2}  \end{matrix} \right), 
 \quad 
\mb{S}_3 = \left( \begin{matrix} 2 & \frac{3}{2} & 0 & 0 \\ 
																				 \frac{3}{2} & 3 & 0 & 0  \\
																			      0 & 0 & 3 & \frac{3}{2} \\
																			      0 & 0 & \frac{3}{2} & 2  	\end{matrix} \right)
\end{align}
Using $\mb{S}_n$ for $n = \{0,1,2,3\}$, the matrix $\mb{S}$ can be given as 
\begin{align}
\mb{S}_{G20} = \left(
\begin{array}{ccccccccccccc}
 \frac{1}{2} & 0 & 0 & 0 & 0 & 0 & 0 & 0 & 0 & 0 & 0 & 0 & 0 \\
 0 & 1 & 0 & 0 & 0 & 0 & 0 & 0 & 0 & 0 & 0 & 0 & 0 \\
 0 & 0 & \frac{1}{2} & 0 & 0 & 0 & 0 & 0 & 0 & 0 & 0 & 0 & 0 \\
 0 & 0 & 0 & 2 & \frac{3}{2} & 0 & 0 & 0 & 0 & 0 & 0 & 0 & 0 \\
 0 & 0 & 0 & \frac{3}{2} & 3 & 0 & 0 & 0 & 0 & 0 & 0 & 0 & 0 \\
 0 & 0 & 0 & 0 & 0 & \frac{1}{2} & 0 & 0 & 0 & 0 & 0 & 0 & 0 \\
 0 & 0 & 0 & 0 & 0 & 0 & \frac{1}{2} & 0 & 0 & 0 & 0 & 0 & 0 \\
 0 & 0 & 0 & 0 & 0 & 0 & 0 & \frac{1}{2} & 0 & 0 & 0 & 0 & 0 \\
 0 & 0 & 0 & 0 & 0 & 0 & 0 & 0 & 1 & \frac{1}{2} & 0 & 0 & 0 \\
 0 & 0 & 0 & 0 & 0 & 0 & 0 & 0 & \frac{1}{2} & 1 & 0 & 0 & 0 \\
 0 & 0 & 0 & 0 & 0 & 0 & 0 & 0 & 0 & 0 & \frac{1}{2} & 0 & 0 \\
 0 & 0 & 0 & 0 & 0 & 0 & 0 & 0 & 0 & 0 & 0 & 3 & \frac{3}{2} \\
 0 & 0 & 0 & 0 & 0 & 0 & 0 & 0 & 0 & 0 & 0 & \frac{3}{2} & 2 \\
\end{array}
\right)
\end{align}
\end{exmp}

The ordering of moments considered in the present work, is different as compared to that considered in \cite{Torrilhon2015}. In \cite{Torrilhon2015}, the solution vector contains all the components of a particular tensor clubbed together. If one uses such an ordering for the solution vector then the matrix $\mb{S}$ simply consists of various $\mb{S}_n$ placed on the diagonal. On the other hand, in the present work we have considered the odd components and the even components of all the tensors to be clubbed together; such an ordering of the solution vector helps us in formulating OBCs for our moment system in a easier way. Therefore, $\mb{S}$ consists of various entries of $\mb{S}_n$ placed at appropriate locations.

\subsection{Maxwell's Accommodation Model}
Let $\td{f}_{\mcal{M}}$ represent the deviation of $f_{\mcal{M}}$ from $f_0$ upto $\mcal{O}(\epsilon)$. Then using our basis functions defined in \eqref{def psi}, $\td{f}_{\mcal{M}}$ can be expressed as 
\begin{align}
\td{f}_{\mcal{M}}\left(\boldsymbol{\xi};\alpha^{(0)},\alpha_i^{(0)},\alpha^{(1)}\right) = 
&f_0\left(\alpha^{(0)}\psi^{(0)} +  \alpha_i^{(0)} \psi_{i}^{(0)} + \alpha^{(1)}\psi^{(1)}\right) \label{fM}
\end{align}
where the coefficients $\alpha$'s are related to the deviation of $\rho$, $v_i$ and $\theta$ through \eqref{relation alpha prim1}. We will now consider a wall such that the normal pointing from the gas into the wall points in the positive-$x$ direction. Let $\hat{f}$ represent the deviation of the distribution function from $f_0$, at the wall, upto $\mcal{O}(\epsilon)$. Then as per the Maxwell's accommodation model, $\hat{f}$ is given as 
\begin{gather}
\hat{f} =  \begin{cases} 
     \chi f_w + (1-\chi)\td{f}_h(\boldsymbol{\xi}^*)  & \xi_1 \leq 0 \\
	 \td{f}_h(\boldsymbol{\xi}) & \xi_1 > 0 \\
   \end{cases} \label{Maxwell acco}
\end{gather}
where $f_w  = \td{f}_{\mcal{M}}(\boldsymbol{\xi},\alpha_w^{(0)},\alpha^w_i,\alpha_w^{(1)})$ with $\alpha_w^{(0)}$, $\alpha^w_{i}$ and $\alpha_w^{(1)}$ being related to the density, velocity and temperature deviation of the wall respectively. The molecular velocity $\boldsymbol{\xi}^*$ is $\boldsymbol{\xi}$ with the sign for $\xi_1$ reversed, i.e. $\boldsymbol{\xi}^* = (-\xi_1,\xi_2,\xi_3)$.  In a given IBVP, the temperature and the velocity of the wall are given whereas the quantity $\alpha_w^{(0)}$ is computed using mass conservation at the wall which implies that the normal velocity of the gas at the wall should be equal to that of the wall. In the present work, we will assume that the wall has no normal velocity which translates into $\alpha^w_x = 0$. Similar to the odd and even splitting of our distribution function given in \eqref{even odd split f}, we can also split $f_w$ as $f_w = f_w^o + f_w^e$
where 
\begin{align}
f_w^o = f_0\left( \alpha_x^{(0)} \psi_{x}^{(0)}\right),\quad f_w^e = f_0\left(\alpha^{(0)}\psi^{(0)} +  \alpha_y^{(0)} \psi_{y}^{(0)} + \alpha_z^{(0)} \psi_{z}^{(0)} + \alpha^{(1)}\psi^{(1)}\right)
\end{align}
Since we have considered $\alpha^w_x = 0$ thus $f_w^o = 0$. The boundary conditions for our moment system in \eqref{moment system} can now be computed using continuity of fluxes which leads to the following expression after some manipulations; see \cite{Struchtrupbook,Torrilhon2015,Torrilhon2003} for more details 
\begin{align}
\alpha_{i_1\dots i_n}^{(s,o)} = \lan \psi_{i_1\dots i_n}^{(s,o)},\td f_h^o  \ran_{\mbb{R}} = \frac{2\chi}{2-\chi}(\lan \psi_{i_1\dots i_n}^{(s,o)},\td f_h^e  \ran_{\mbb{R}^+} - \lan \psi_{i_1\dots i_n}^{(s,o)},f_w^e  \ran_{\mbb{R}^+}) \label{continuity fluxes}
\end{align}
where $\psi_{i_1\dots i_n}^{(s,o)}$ are the basis functions which are odd with respect to $\xi_1$. Similarly, $\td{f}_h^o$ and $\td{f}_h^e$ represent the odd and even part of the distribution function with respect to $\xi_1$. Due to no-penetration boundary condition at the wall we have 
\begin{align}
\alpha_x^{(0)} = 0.	\label{no penetration}
\end{align}
If we now consider $\psi_{i_1\dots i_n}^{(s,o)}$ to be $\psi_x^{(0)}$ in \eqref{continuity fluxes} and consider the no penetration boundary condition given in the above expression, we obtain the following expression for $\alpha_w^{(0)}$
\begin{align}
\alpha_w^{(0)} = \frac{\left(\lan\psi_x^{(0)},\td{f}_h^e\ran_{\mbb{R}^+}-\alpha_w^{(1)}\lan\psi_x^{(0)},\psi^{(1)}\ran_{(\mbb{R}^+,f_0)}\right )}{\lan\psi_x^{(0)},\psi^{(0)}\ran_{(\mbb{R}^+,f_0)}}
\end{align}
In writing the above expression we have used 
\begin{align}
\lan\psi_x^{(0)},\psi_{y}^{(0)}\ran_{(\mbb{R}^+,f_0)} = \lan\psi_x^{(0)},\psi_{z}^{(0)}\ran_{(\mbb{R}^+,f_0)} = 0.
\end{align}
Substituting the above relation for $\alpha_w^{(0)}$ into \eqref{continuity fluxes} we obtain the following MBCs
\begin{align}
\boldsymbol{\alpha}_o = 2\beta \mb{M}^{(mbc)}\boldsymbol{\alpha}_e + 2\beta\mb{g}	\label{MBC}
\end{align}
where $\beta = \chi/(2-\chi)$, the matrix $\mb{M}^{(mbc)}\in \mbb{R}^{n_o\times n_e}$. The vector $\mb{g}\in \mbb{R}^{n_o}$ is the inhomogeneity arising from the wall. Let $\mb{m}_{\alpha^{(1)}}$, $\mb{m}_{\alpha^{(0)}_y}$ and $\mb{m}_{\alpha^{(0)}_z}$ represent those columns of $\mb{M}^{(mbc)}$ which are multiplied by $\alpha^{(1)}$, $\alpha^{(0)}_y$ and $\alpha^{(0)}_z$ respectively then the vector $\mb{g}$ can be given as 
\begin{align}
\mb{g} = -\left(\alpha_w^{(1)}\mb{m}_{\alpha^{(1)}} + \alpha_y^{w}\mb{m}_{\alpha^{(0)}_y}+ \alpha_z^{w}\mb{m}_{\alpha^{(0)}_z}\right)	\label{def g}
\end{align}
As discussed above, we have assumed that the wall has zero velocity in the normal direction therefore due to the no-penetration boundary condition given in \eqref{no penetration} we find that the first entry of $\mb{g}$ will be zero i.e.
\begin{align}
g_1 = 0.	\label{structure g}
\end{align}
Due to the computation of $\alpha_w^0$ we note that the matrix $\mb{M}^{(mbc)}$ will have the following structure 
\begin{align}
\mb{M}^{(mbc)} = \left( \begin{matrix} \mb{0} & \mb{0} \\
						\mb{0} & \td{\mb{M}}^{(mbc)} \end{matrix} \right).\label{structure MBC}
\end{align}

The above structure of $\mb{M}^{(mbc)}$ shows us that the denisty of the fluid, $\alpha^{(0)}$, does not influence any of the boundary conditions. A similar structure as above for $\mb{M}^{(mbc)}$ was also identified in \cite{Sarna2017} and was helpful in proving the stability of inhomogeneous OBCs. Having formulated the MBCs, we can now study their stability using the conditions given in \eqref{stability condition inhomo}. Using computational analysis we have found that for all the systems, from $G10$ to $G148$, the MBCs are not stable. Since MBCs do not provide us with a stable set of boundary conditions thus we will now look for a set of boundary conditions which are stable.  

\subsection{Onsager Boundary Conditions}
In order to formulate OBCs for our moment system, we need to find similarities between our moment system in \eqref{moment system} and the general formulation developed in \autoref{onsager boundary conditions}. Therefore we will first look into the structure of $\mb{S}\mb{A}^{(1)}$. From the structure of $\mb{A}^{(1)}$ given in \eqref{structure on A1} and the fact that SA is symmetric, we find 
\begin{align}
\mb{S}\mb{A}^{(1)} =  \left( \begin{matrix}\mb{0} &  \mb{A}^{oe}\\
									\left(\mb{A}^{oe}\right)^T&  0 	\end{matrix} \right)	\label{structure on An moment}
\end{align}
Since the general formulation of OBCs relied upon certain assumptions made upon the properties of $\mb{A}^{(1)}$ so we will be assuming the following for the matrix $\mb{S}\mb{A}^{(1)}$ and $\mb{A}^{(1)}$ corresponding to our moment system 
\begin{itemize}
\item the total number of negative eigenvalues of $\mb{A}^{(1)}$ are equal to the total number of odd variables in the system i.e. $n_o$.
\item the $\ker\{\mb{A}^{(1)}\}$ has the following structure 
\begin{align}
ker\{\mb{S}\mb{A}^{(1)}\} = ker\{\mb{A}^{(1)}\}=\left( \begin{matrix} \mb{0}\\
									ker\{\mb{A}^{oe}\} \end{matrix}
									 \right)	\label{structure kernel An}									 
\end{align}
\item splitting the matrix $\mb{A}^{oe}$ as 
\begin{align}
\mb{A}^{oe} =\left(\hat{\mb{A}}^{oe}, \tilde{\mb{A}}^{oe}\right)	\label{split Aoe}
\end{align} 
where $\hat{\mb{A}}^{oe}\in\mbb{R}^{n_o\times n_o}$ and $\td{\mb{A}}^{oe}\in \mbb{R}^{n_o\times (n_e-n_o)}$. We will assume $\hat{\mb{A}}^{oe}$ to be invertible. 
\end{itemize}
The above two assumptions have already been discussed in \autoref{onsager boundary conditions}, the motivation behind the third assumption becomes clear once we consider our model for the Onsager matrix $\mb{L}$. The assumptions described above do not put any restriction upon the applicability of the boundary conditions to be presented because we have found through numerical studies that all the moment systems from $G10$ to $G148$ satisfy the above assumptions. 
Having made the necessary assumptions, we can now compare the solution vector of our moment system (see \eqref{moment system}), $\boldsymbol{\alpha}$, and the matrix $\mb{S}\mb{A}^{(1)}$  with those presented in \autoref{onsager boundary conditions}. This leads to

\begin{align}
\boldsymbol{\alpha}_o = \boldsymbol{\alpha}_p,\quad \boldsymbol{\alpha}_e = \boldsymbol{\alpha}_q\quad \mb{A}_{oe} = \mb{A}^*
\end{align}
Using the above relations in our general OBCs given in \eqref{Onsager boundary conditiones inhomo}, we obtain the following set of OBCs for our moment system 
\begin{align}
\boldsymbol{\alpha}_o = \mb{L}\mb{A}^{oe}\boldsymbol{\alpha}_e + 2\beta\mb{g}	\label{OBC moment system}
\end{align}
where $\mb{L}\in \mbb{R}^{n_o\times n_o}$ is an unknown symmetric positive semi-definite matrix and $\mb{g}$ is as defined in \eqref{def g}. Before considering the explicit expression for $\mb{L}$, it would be helpful to consider the following decomposition for $\mb{M}^{(mbc)}$
\begin{align}
\mb{M}^{(mbc)} =\left(\hat{\mb{M}}^{(mbc)}, \tilde{\mb{M}}^{(mbc)}\right)	\label{split Bmbc}
\end{align} 
where $\hat{\mb{M}}^{(mbc)}\in \mbb{R}^{n_o\times n_o}$ and $ \tilde{\mb{M}}^{(mbc)}\in \mbb{R}^{n_o\times \left(n_e-n_o\right)}$. Adopting the model for the Onsager matrix proposed in \cite{Sarna2017}, we have the following explicit expression for $\mb{L}$
\begin{align}
\mb{L} = 2\beta\hat{\mb{M}}^{(mbc)}\left(\hat{\mb{A}}^{oe}\right)^{-1}. \label{model R}
\end{align}

From the above model we can see that the invertibility of $\left(\hat{\mb{A}}^{oe}\right)$ is crucial if we wish to extend the framework developed in \cite{Sarna2017} for multi-dimensional problems. In the present work we will not be proving that an Onsager matrix given by \eqref{model R} will be symmetric positive semi-definite. But through a numerical study, for $G10$ to $G148$, we have found that even for multi-dimensional moment systems, the Onsager matrix $\mb{L}$ given by \eqref{model R} is symmetric positive semi-definite. 
With an explicit expression for the Onsager matrix, a set of stable boundary conditions for a general wall, with $\mb{n}$, $\mb{t}$ and $\mb{r}$ spanning it's local coordinate system, is given as
\begin{align}
\boldsymbol{\alpha}_o^{(n,t,r)} = \mb{L}\mb{A}^{oe}\boldsymbol{\alpha}_e^{(n,t,r)} + 2\beta\mb{g}	\label{OBC moment system general wall}
\end{align}
We can now analyse the stability of our OBCs given in \eqref{OBC moment system} through the following way. Using the structure of $\mb{M}^{(mbc)}$ given in \eqref{structure MBC}, we find that our Onsager matrix given in \eqref{model R} will have the following structure 
\begin{align}
\mb{L} = \left( \begin{matrix} \mb{0} & \mb{0} \\
						\mb{0} & \td {\mb{L}} \end{matrix} \right).\label{structure MBC}
\end{align}
Using \eqref{structure g} and the structure of our moment system, we find that the vector $\mb{g}$ and the matrix $\mb{A}_{oe}$ will have the following form
\begin{align}
\mb{g} = \left( \begin{matrix} \mb{0}\\
					\td{\mb{g}} \end{matrix} \right),\quad \mb{A}_{oe} =	 \left( \begin{matrix} 1 & \dots\\
						\mb{0} & \mb{A}^{\dagger}_{oe} \end{matrix} \right)	\label{structure g and Aoe}
\end{align}
A similar structure for $\mb{L}$, $\mb{g}$ and $\mb{A}_{oe}$ was also recognised in \cite{Sarna2017} and similar to the one studied in \eqref{onsager boundary conditions}; therefore the OBCs given in \eqref{OBC moment system general wall} along with the Onsager matrix given in \eqref{structure MBC} provides us with a stable set of boundary conditions.
\section{Poisson Heat Conduction}
Having formulated a stable set of boundary conditions for our moment system, we would now like to compare the physical accuracy provided by the newly proposed OBCs with respect to the MBCs. To achieve this, we will revisit the steady state Poisson heat conduction problem studied in \cite{Torrilhon2015} where the author has discussed the convergence behaviour of higher order moment methods for boundary value problems using MBCs. Moving along the same lines as in \cite{Torrilhon2015}, we will approximate our linearised collision operator, $\td Q\td f$, appearing in the linearised Boltzmann's equation \eqref{BE linear} through the BGK model which is given as 
\begin{align}
\td Q (\td f) = -\frac{1}{\tau}\left(\td f-\td f_{\mcal{M}}\right)
\end{align}
where $\tau$ represents the relaxation time scale and is the inverse of the collision frequency. 
To study the Poisson heat conduction problem, we will consider a channel which extends infinitely in the $x$-direction. So, all the field variable will only vary along the $y$-direction. The channel will be considered to be symmetric about the $x$-axis such that $y \in [-\frac{L}{2},\frac{L}{2}]$. The flow will be characterised by the Knudsen number $Kn$ which is given as
\begin{align}
Kn = \frac{\tau}{\sqrt{\theta_0}L}
\end{align}
Additionally, we will introduce a source term $F(\mb{x})$ on the right hand side of the linearised Boltzmann's equation \eqref{BE linear}. The forcing term, $F(\mb{x})$, will be such that it only influences the energy equation (or the equation for $\alpha^{(1)}$); therefore we will consider $F(\mb{x})$ to be given by 
\begin{align}
F(\boldsymbol{\xi},\mb{x}) = -\sqrt{\frac{2}{3}} \frac{r(\mb{x})}{\rho_0\theta_0}\psi^{(1)}f_0(\boldsymbol{\xi})
\end{align}
where $r(\mb{x})$ is some function of $\mb{x}$. If we consider the above form for the source term, then our energy equation reads (in steady state)
\begin{align}
\rho_0\theta_0\pd_{x_i}\td v_i + \pd_{x_i}\td q_i = r	\label{energy equation Poisson}
\end{align}
All the other equations see no influence from $F$ due to the orthogonality of the basis functions given in \eqref{orthogonality psi}. For a detailed discussion on the Poisson heat conduction problem see \cite{Torrilhon2015}. Since we choose to drive our system with the help of an external force so we will be considering both the walls of the channel to be at the same temperature and stationary. 

\subsection{Problem Setup}
In the present work, we will consider the following functional form for our source term $r(\mb{x})$ appearing the energy equation
\begin{align}
r(y) = \alpha y^2
\end{align}
where $\alpha = \sqrt{\frac{2}{3}}$. Since we have considered a forcing term which is symmetric with respect to the $x$-axis so all the field variables will be symmetric about the same. We will scale the $y$-coordinate with $L$ and will use appropriate powers of $\rho_0$ and $\theta_0$ to scale all the other macroscopic quantities like velocity, stress tensor, heat flux etc. In order to fully define our boundary conditions, we will need the attributes of the wall which are given as
\begin{subequations}
\begin{align}
\alpha_w^{(1)}\rvert_{y = -\frac{1}{2}} = \alpha_w^{(1)}\rvert_{y = \frac{1}{2}} = -\sqrt{\frac{3}{2}},\\ \alpha_x^{w}\rvert_{y = -\frac{1}{2}} = \alpha_x^{w}\rvert_{y = \frac{1}{2}} = 0
\end{align}
\end{subequations}
To maintain consistency with the work done in \cite{Torrilhon2015}, we will be considering $Kn=0.3$. The rarefaction effects becomes important for $Kn \geq 0.05$; this is the regime where the classical Navier-Stokes equations fail to provide us with an acceptable solution. These rarefaction effects include a temperature jump condition at the wall, a non-trivial stress-tensor etc. Since the moment systems have been found to have an oscillatory convergence behaviour for boundary value problems so similar to \cite{Torrilhon2017} we will be using the averaged solution of $G56$, $G84$ and $G120$ moment equations to study the Poisson heat conduction problem. For details regarding the reference solution see \cite{Torrilhon2015}.

\subsection{Variation of field variables}
In \autoref{field variation}, we have shown the variation of $\td{\theta}$ and $\td{\sigma}_{yy}$ along with the variation of $e_{\theta}$ and $e_{\sigma}$ which are defined as 
\begin{align}
e_{\theta}(y) = |\td\theta-\theta_{ref}|,\quad e_{\sigma}(y) = |\td\sigma_{yy}-\sigma_{yy}^{(ref)}|
\end{align}
where $\theta_{ref}$ and $\sigma_{yy}^{(ref)}$ represent the reference $\td{\theta}$ and $\td{\sigma}_{yy}$ respectively. 
Let us first look into the variation of $\td{\theta}$. 
As one would expect from this particular flow regime, we see a temperature jump at the wall while using both OBCs and MBCs. 
Considering the physical accuracy, we can see that a few mean free paths away from the wall, the solution obtained through OBCs provides us with much more accurate results as compared to the MBCs. As we move closer to the wall, both OBCs and MBCs fail to  capture the sharp boundary layer; though the results obtained from OBCs appear to be qualitatively more appropriate. This shows us that in order to capture the boundary layer more precisely one needs to consider even higher order moment methods. 
 A similar observation, in relation to R13 equations, was also made in \cite{Rana2016}. 

\begin{figure}[ht!]
\centering
\subfigure [Variation of $\td\theta$ for $Kn=0.3$]{
\includegraphics[width=3in]{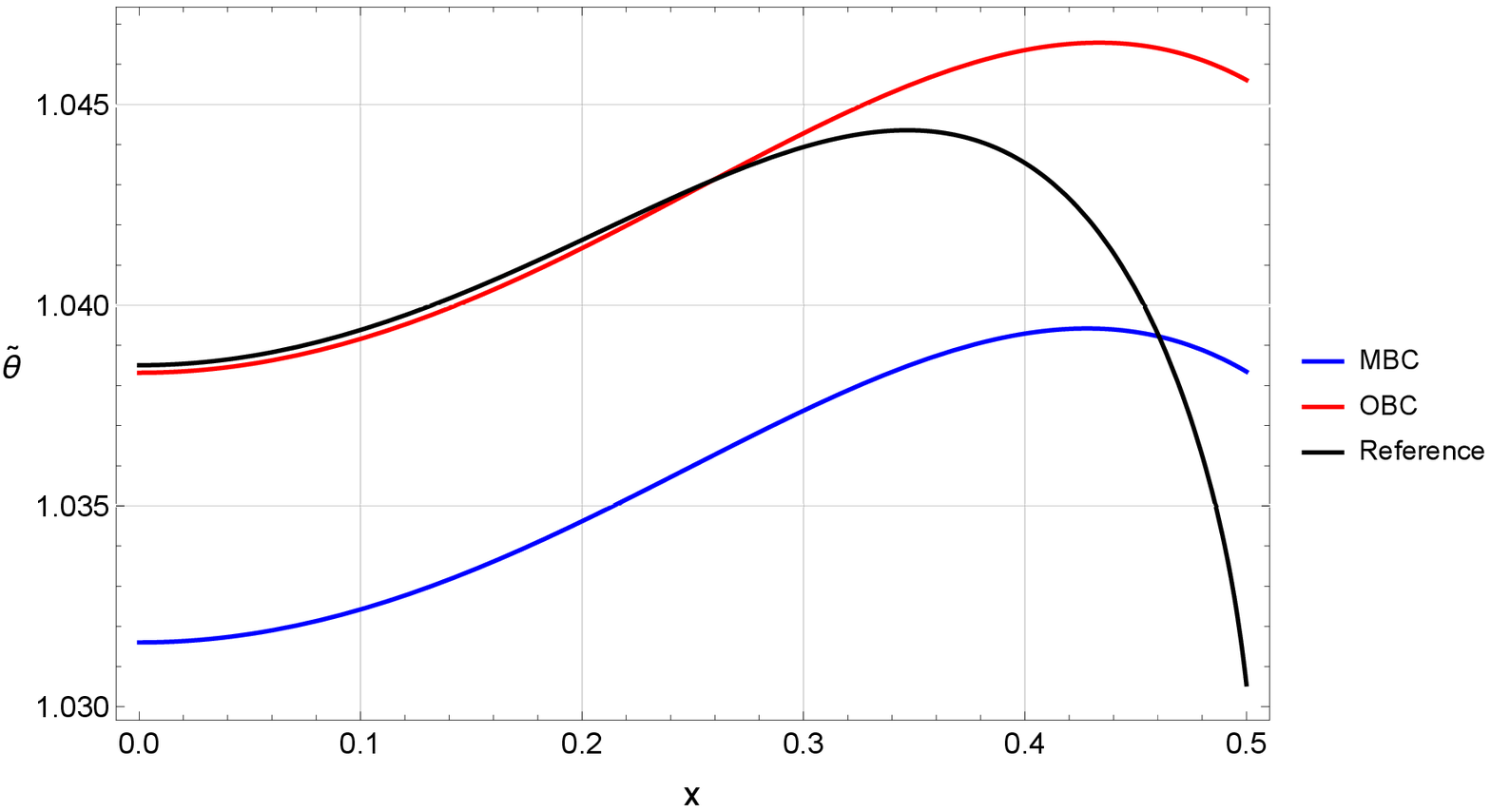} }
\hfill
\subfigure [Variation of $\td\sigma_{yy}$ for $Kn=0.3$]{
\includegraphics[width=3in]{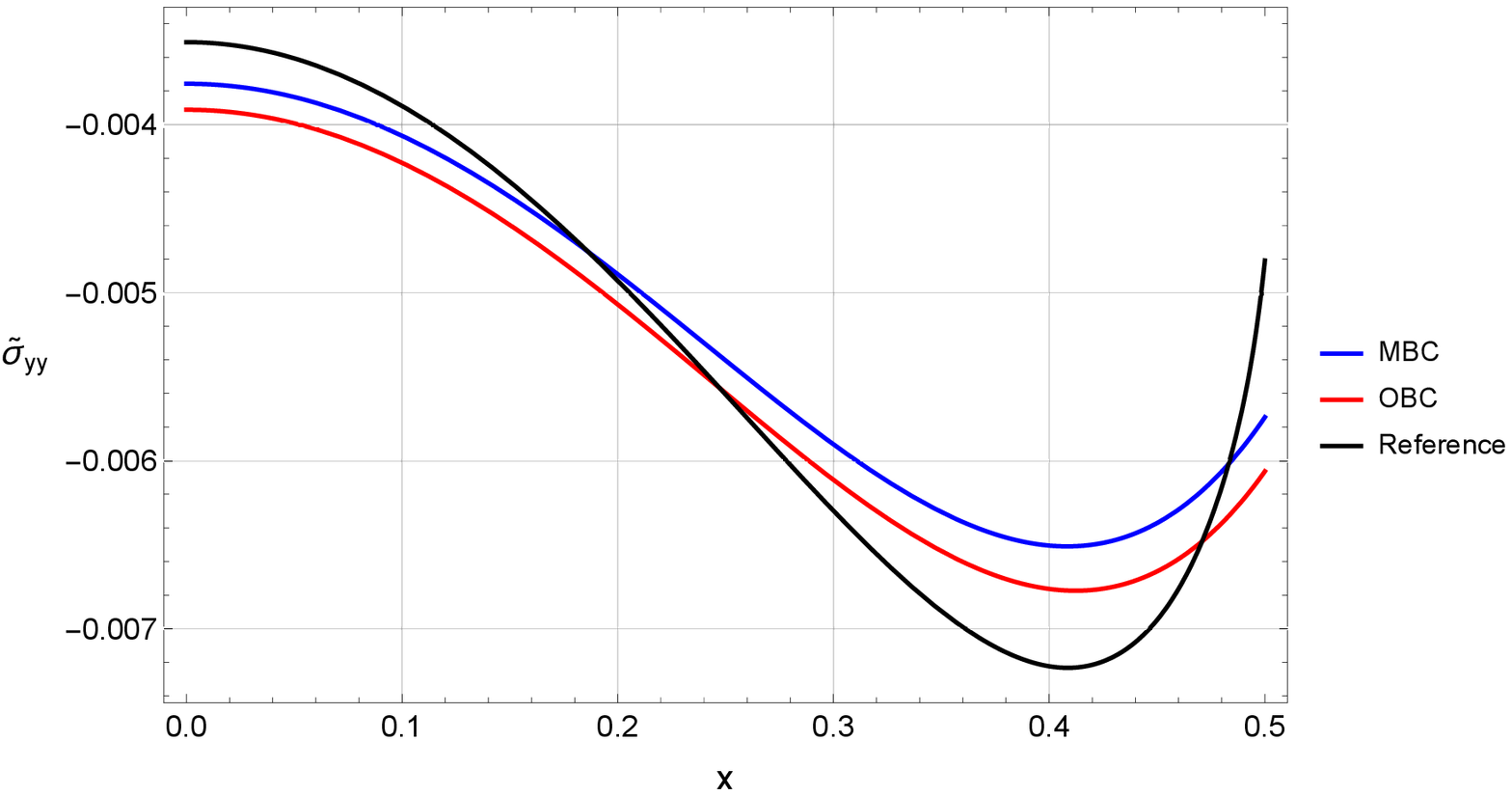} }
\hfill
\subfigure [Variation of error in $\td\theta$ for $Kn=0.3$]{
\includegraphics[width=3in]{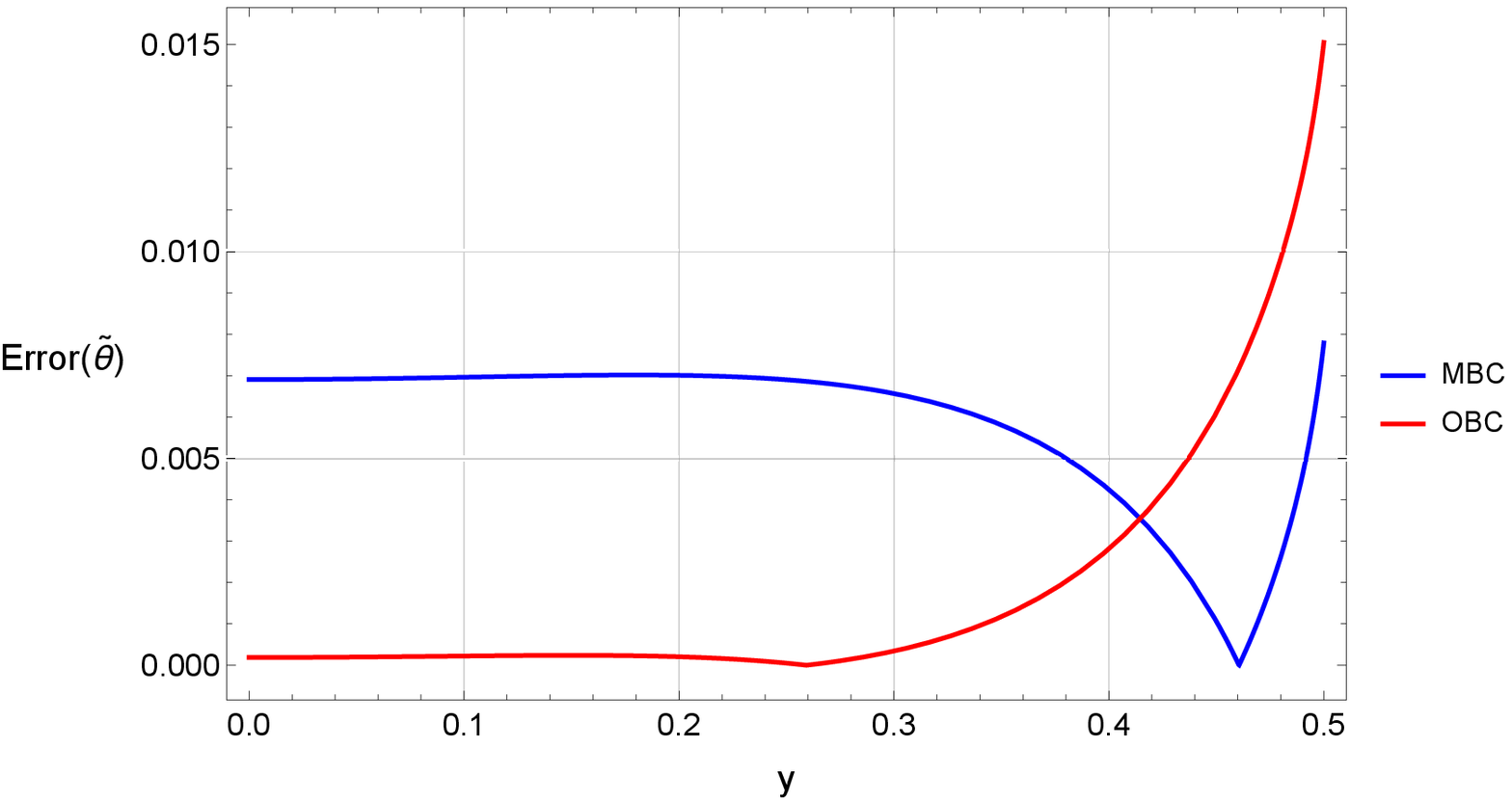} }
\hfill
\subfigure [Variation of error in $\td\sigma_{yy}$ for $Kn=0.3$]{
\includegraphics[width=3in]{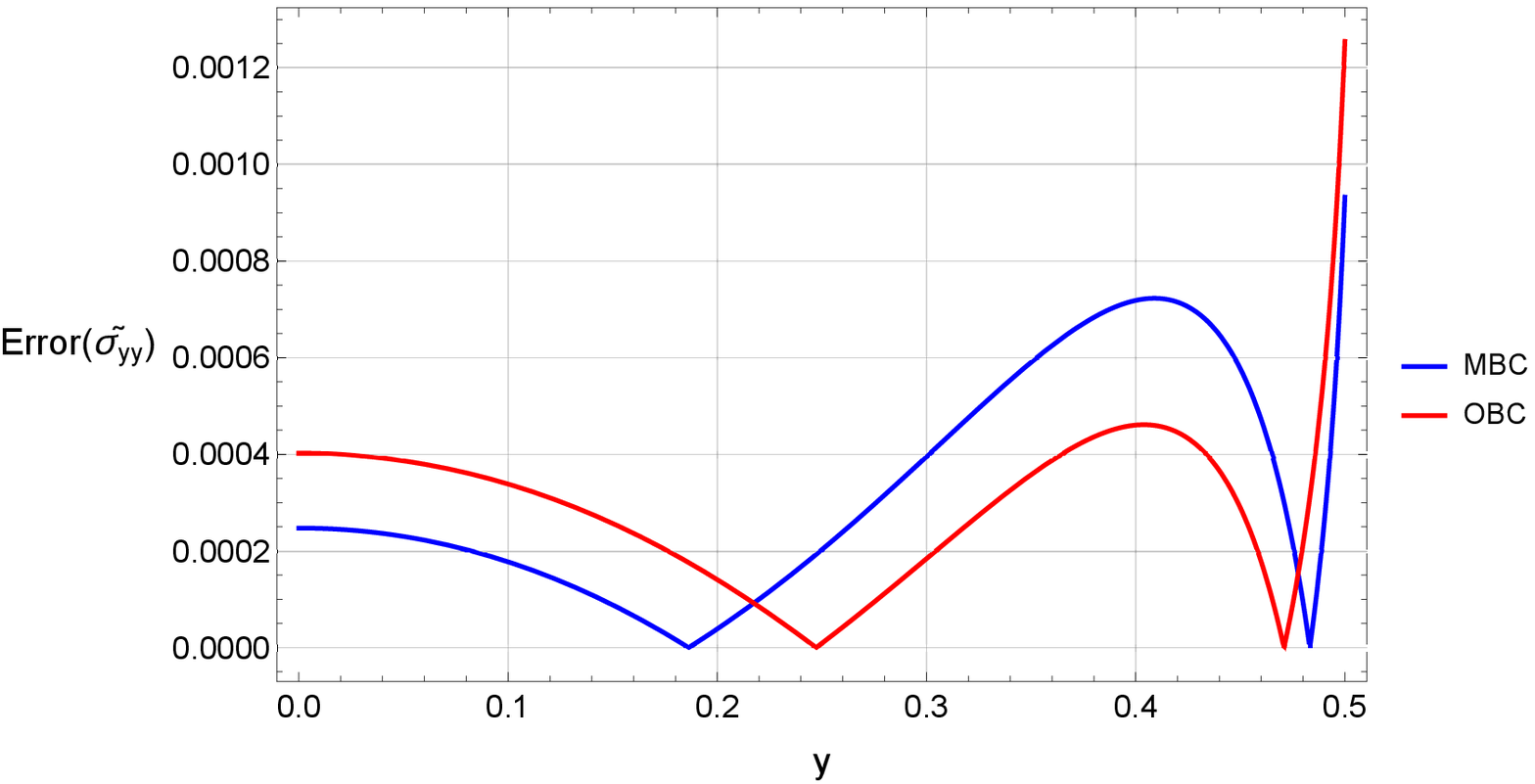} }
\hfill
\caption{Variation of the field variables $\td\theta$ and $\td\sigma_{yy}$ and their corresponding error for MBCs and OBCs using the averaged solution of $G56$, $G84$ and $G120$ moment equations }	\label{field variation}
\end{figure} 

We can now consider the variation of  $\td{\sigma}_{yy}$. Similar to the temperature jump effects seen in the variation of $\td{\theta}$, we see a non-trivial $\td{\sigma}_{yy}$ in the channel which is a well known rarefaction effect. Contrary to the variation of $\td{\theta}$, the results obtained for $\td{\sigma}_{yy}$, using MBCs or OBCs are very similar qualitatively. The variation of error for $\td{\sigma}_{yy}$ follows a sporadic behaviour. Near the central axis of the channel, the results obtained from MBCs are more accurate but as we move closer to the wall they become less accurate as compared to OBCs only to become more accurate very close to the wall. Similar to $\td{\theta}$, the variation of $\td{\sigma}_{yy}$ shows us that we need to consider more moments in order to capture the boundary layer accurately. 

\section{Conclusion}
We have used the symmetric hyperbolicity and the rotational invariance of the linear moment systems to come up with stable boundary conditions for the same. The stable boundary conditions were formulated in terms of an unknown Onsager matrix $\mb{L}$ which was then defined using the model presented in \cite{Sarna2017}. In order to extend the model for the Onsager matrix presented in \cite{Sarna2017}, to the multi-dimensional case, we have made certain assumption on the properties of the flux matrices. The assumptions on these properties were found to hold true even for very large moment systems and therefore the framework presented in this work is not restricted to only certain moment systems. Using the properties of the Onsager matrix and the flux matrices, the boundary conditions were shown to be stable even for the inhomogeneous case. To compare the physical accuracy of the MBCs and the OBCs we have revisited the Poisson heat conduction problem studied in \cite{Torrihon2015}. For this particular test case, both MBCs and OBCs were found to be inaccurate very close to the boundary of the domain but the OBCs were found to be more accurate in the bulk region. Any realistic flow computation is a combination of various flow phenomenons one of which is heat conduction therefore  a definitive answer regarding physically accuracy of the OBCs could not be made only by the analysis done in the present work.

\bibliographystyle{apa}
\bibliography{paper}

\end{document}